\documentclass[11pt]{article}
\usepackage{amsmath}
\usepackage{amsfonts}
\usepackage{amssymb}

\usepackage{graphicx}
\usepackage{eso-pic}

\renewcommand{\Box}{\framebox{\rule{0.3em}{0.0em}}}

\newtheorem{thm}{Theorem}[section]
\newtheorem{lema}{Lemma}[section]

\newcommand{\bgeqn}{\begin{eqnarray}}
\newcommand{\edeqn}{\end{eqnarray}}
\newcommand{\bgeq}{\begin{eqnarray*}}
\newcommand{\edeq}{\end{eqnarray*}}
\newcommand{\bec}{\begin{center}}
\newcommand{\enc}{\end{center}}

\renewcommand{\Box}{\hfill \rule{2.3mm}{2.3mm}}

\makeatletter

\@addtoreset{equation}{section} \makeatother
\newenvironment{proof}{\noindent{\bf Proof. }}{\hfill $\Box$\medskip}

\title{The log-convexity of the poly-Cauchy numbers}

\author{Takao Komatsu
\thanks{
The first author was supported in part by the grant of Wuhan University and by the grant of Hubei Provincial Experts Program.
}
\\
\small School of Mathematics and Statistics\\[-0.8ex]
\small Wuhan University, Wuhan, 430072, China\\[-0.8ex]
\small \texttt{komatsu@whu.edu.cn}\\\\
Feng-Zhen Zhao
\thanks{
The second author was supported by a grant of ``The First-class Discipline of Universities in Shanghai".
}\\
\small Department of Mathematics\\[-0.8ex]
\small Shanghai University\\[-0.8ex]
\small Shanghai 200444, China\\[-0.8ex]
\small \texttt{fengzhenzhao@shu.edu.cn}
}

\date{
}

\begin{document}
\maketitle

\baselineskip 16pt

\begin{abstract}  In 2013, Komatsu introduced the poly-Cauchy numbers, which generalize Cauchy numbers. Several generalizations of poly-Cauchy numbers have been considered since then. One particular type of generalizations is that of multiparameter-poly-Cauchy numbers.  In this paper, we study the log-convexity of the multiparameter-poly-Cauchy numbers of the first kind and of the second kind.
In addition, we also discuss the log-behavior of multiparameter-poly-Cauchy numbers.
\end{abstract}

\section{Introduction}

Komatsu \cite{ref10}, introduced two kinds of poly-Cauchy numbers $c_n^{(k)}$ and $\widehat{c}_n^{(k)}$. The first kind $c_n^{(k)}$ is given by
\begin{eqnarray*}
c_n^{(k)}&=&\underbrace{\int_0^1\cdots \int_0^1}_k (x_1x_2\cdots x_k)_n dx_1dx_2\cdots dx_k,
\end{eqnarray*}
and the second kind $\widehat{c}_n^{(k)}$ is given by
\begin{eqnarray*}
\widehat{c}_n^{(k)}&=&(-1)^n\underbrace{\int_0^1\cdots \int_0^1}_k \langle x_1x_2\cdots x_k\rangle_n dx_1dx_2\cdots dx_k,
\end{eqnarray*}
where $k$ is a positive integer, and
$(x)_n=x(x-1)\cdots(x-n+1$ ($n\geq 1$) with $(x)_0=1$ and
$\langle x \rangle_n=x(x+1)\cdots(x+n-1)$ with $\langle x \rangle_0=1$.
The first few values of $c_n^{(k)}$ and $\widehat{c}_n^{(k)}$ are
\begin{eqnarray*}
&&c_0^{(k)}=1, c_1^{(k)}=\frac{1}{2^k},  c_2^{(k)}=-\frac{1}{2^k}+\frac{1}{3^k}, c_3^{(k)}=\frac{2}{2^k}-\frac{3}{3^k}+\frac{1}{4^k}, c_4^{(k)}=-\frac{6}{2^k}+\frac{11}{3^k}-\frac{6}{4^k}+\frac{1}{5^k}, \cdots \\
&&\widehat{c}_0^{(k)}=1, \widehat{c}_1^{(k)}=-\frac{1}{2^k}, \widehat{c}_2^{(k)}=\frac{1}{2^k}+\frac{1}{3^k}, \widehat{c}_3^{(k)}=-\frac{2}{2^k}-\frac{3}{3^k}-\frac{1}{4^k}, \widehat{c}_4^{(k)}=\frac{6}{2^k}+\frac{11}{3^k}+\frac{6}{4^k}+\frac{1}{5^k}, \cdots
\end{eqnarray*}
Some sequences derived from denominators and numerators of poly-Cauchy numbers of the first kind and of the second kind can be found in \cite[A224094--A224101,A219247,A224102--A224107,A224109]{oeis}.
When $k=1$, $c_n=c_n^{(1)}$ and $\widehat c_n=\widehat{c}_n^{(1)}$ are the Cauchy numbers of the first kind and the second kind, respectively (see \cite{ref3}).  The basic properties of the two kinds of the poly-Cauchy numbers are studied in \cite{ref9, ref10}. Several generalizations of poly-Cauchy numbers have been considered since then.  One particular type of generalizations is that of the multiparameter-poly-Cauchy numbers (\cite{ks}).
For a $k$-tuple of real numbers $L=(l_1,\dots,l_k)$ and a $n$-tuple of real numbers $A=(\alpha_0,\alpha_1,\dots,\alpha_{n-1})$,
define the $q$-multiparameter-poly-Cauchy polynomials of the first kind $c_{n,L,A,q}^{(k)}(z)$ by
$$
c_{n,L,A,q}^{(k)}(z)=\int_0^{l_1}\cdots\int_0^{l_k}(x_1\cdots x_k-\alpha_0-z)\cdots(x_1\cdots x_k-\alpha_{n-1}-z)d_q x_1\cdots d_q x_k\,.
$$
Here, Jackson's $q$-integral is defined by
$$
\int_0^x f(t)d_q t=(1-q)x\sum_{n=0}^\infty f(q^n x)q^n\,,
$$
and
$$
[x]_q=\frac{1-q^x}{1-q}\to x\quad(q\to 1)\,.
$$
The $q$-multiparameter-poly-Cauchy polynomials of the first kind can be expressed explicitly in terms of the multiparameter Stirling numbers of the first kind $S_1(n,m,A)$, defined by
$$
(t-\alpha_0)(t-\alpha_1)\cdots(t-\alpha_{n-1})=\sum_{m=0}^n S_1(n,m,A)t^m\,.
$$
Namely, we have
$$
c_{n,L,A,q}^{(k)}(z)=\sum_{m=0}^n S_1(n,m,A)\sum_{i=0}^m\binom{m}{i}\frac{(-z)^i(l_1\cdots l_k)^{m-i+1}}{[m-i+1]_q^k}
$$
(\cite[Theorem 1]{ks}).

When $q\to 1$, $l_1=\cdots=l_k=1$ and $z=0$, we have
$$
c_{n,A}^{(k)}=\lim_{q\to 1}c_{n,(1,\dots,1),A,q}^{(k)}(0)
=\sum_{m=0}^n\frac{S_1(n,m,A)}{(m+1)^k}\,.
$$
Furthermore, if  $A=(0,1,\dots,n-1)$, then $S_1(n,m)=(-1)^{n-m}S_1(n,m,A)$ are the unsigned Stirling numbers of the first kind and $c_n^{(k)}=c_{n,(0,1,\dots,n-1)}^{(k)}$ are poly-Cauchy numbers of the first kind.

Similarly, define the $q$-multiparameter-poly-Cauchy polynomials of the second kind $\widehat c_{n,L,A,q}^{(k)}(z)$ by
$$
\widehat c_{n,L,A,q}^{(k)}(z)=\int_0^{l_1}\cdots\int_0^{l_k}(-x_1\cdots x_k-\alpha_0+z)\cdots(-x_1\cdots x_k-\alpha_{n-1}+z)d_q x_1\cdots d_q x_k\,.
$$
When $q\to 1$, $l_1=\cdots=l_k=1$,  $\alpha_i=i\rho$ ($i=0,1,\dots,n-1$) and $z=0$, the number $\widehat c_{n,A}^{(k)}$ is reduced to the poly-Cauchy numbers of the second kind with a parameter $\rho$ (\cite{ref10a}).  Furthermore, if $\rho=1$, then the number $\widehat c_{n,A}^{(k)}$ is reduced to the poly-Cauchy numbers of the second kind $\widehat c_n^{(k)}$ (\cite{ref10}).
If $k=1$, then $\widehat c_n^{(1)}=\widehat c_n$ is the classical Cauchy number (\cite{ref3}).
The $q$-multiparameter-poly-Cauchy polynomials of the second kind can be also expressed explicitly in terms of the multiparameter Stirling numbers of the first kind as follows.
$$
\widehat c_{n,L,A,q}^{(k)}(z)=\sum_{m=0}^n(-1)^m S_1(n,m,A)\sum_{i=0}^m\binom{m}{i}\frac{(-z)^i(l_1\cdots l_k)^{m-i+1}}{[m-i+1]_q^k}
$$
(\cite[Theorem 2]{ks}).
When $q\to 1$, $l_1=\cdots=l_k=1$ and $z=0$, we have
$$
\widehat c_{n,A}^{(k)}=\lim_{q\to 1}\widehat c_{n,(1,\dots,1),A,q}^{(k)}(0)
=\sum_{m=0}^n\frac{(-1)^m S_1(n,m,A)}{(m+1)^k}\,.
$$

In \cite{ref15}, the log-convexity of Cauchy numbers of the first kind and of the second kind has been studied.
We recall some other definitions and notations used in this paper.

Let $\{z_n\}_{n\geq 0}$ be a
sequence of positive numbers. If for all $j\geq 1$, $z_j^2\leq
z_{j-1}z_{j+1}$ (respectively $z^2_j\ge z_{j-1}z_{j+1}$), the
sequence $\{z_n\}_{n\geq 0}$ is called log-convex (respectively
log-concave).

The log-behavior (log-convexity and log-concavity) are important properties of combinatorial sequences, and they play an important role in many subjects such as quantum physics, white noise theory, probability, economics and mathematical biology. See for instance \cite{ref1, ref2, ref4, ref5, ref6, ref8, ref12, ref13, ref14}.

If $z_0\leq z_1\leq\cdots \leq z_{m-1}\leq z_m\geq z_{m+1}\geq\cdots$ for some $m$, then $\{z_n\}_{n\geq 0}$ is called unimodal, and $m$ is called a mode of the sequence.

In this paper, we study the log-convexity of multiparameter-poly-Cauchy numbers of the first kind and of the second kind.
In addition, we also discuss the log-behavior of multiparameter-poly-Cauchy numbers.

\section{The log-convexity of poly-Cauchy numbers}

In this section, we mainly discuss the log-behavior of $\{c_n^{(k)}\}_{n\ge 2}$, $\{\widehat c_n^{(k)}\}_{n\ge 0}$, $\{c_{n,A}^{(k)}\}_{n\ge 2}$, and $\{\widehat c_{n,A}^{(k)}\}_{n\ge 0}$.
For convenience, put
$$
\sigma_n^{(k)}=(-1)^{n-1}c_n^{(k)}, \quad \sigma_{n,A}^{(k)}=(-1)^{n-1}c_{n,A}^{(k)}\quad(n\ge 1),
$$
and
$$
\omega_n^{(k)}=(-1)^n\widehat{c}_n^{(k)}, \quad \omega_{n,A}^{(k)}=(-1)^n\widehat{c}_{n,A}^{(k)}\quad(n\ge 0).
$$

In \cite{ref15}, the log-convexity of Cauchy numbers was discussed. First, we shall investigate the log-convexity of the poly-Cauchy numbers of the two kinds.

\begin{lema}\cite{ref11}
If $\{y_n\}_{n\ge 0}$ is log-convex, then the Stirling transformation of the first kind $z_n=\sum_{m=0}^n[{n\atop m}]y_m$ preserves the log-convexity.
\end{lema}

\begin{thm}
The sequences $\{c_n^{(k)}\}_{n\ge 2}$ and $\{\widehat c_n^{(k)}\}_{n\ge 0}$ are log-convex.
\end{thm}
\begin{proof}
We first prove the log-convexity of $\{\sigma_n^{(k)}\}_{n\ge 2}$. For $n\ge 1$,
\begin{eqnarray*}
&&\quad\bigg(c_n^{(k)}\bigg)^2-c_{n-1}^{(k)}c_{n+1}^{(k)}\\
&&=\bigg[\underbrace{\int_0^1\cdots \int_0^1}_k x_1\cdots x_k(1-x_1\cdots x_k)\cdots(n-1-x_1\cdots x_k)dx_1\cdots dx_k\bigg]^2\\
&&\quad-\underbrace{\int_0^1\cdots \int_0^1}_k x_1\cdots x_k(1-x_1\cdots x_k)\cdots(n-2-x_1\cdots x_k)dx_1d\cdots dx_k\\
&&\quad\times\underbrace{\int_0^1\cdots \int_0^1}_k x_1\cdots x_k(1-x_1\cdots x_k)\cdots(n-x_1\cdots x_k)dx_1\cdots dx_k.
\end{eqnarray*}
For $0\le x_j\le 1$ ($1\le j\le k$), $n-1\le n-x_1x_2\cdots x_k\le n$. Then for $n\ge3$,
\begin{eqnarray*}
&&\quad\bigg(c_n^{(k)}\bigg)^2-c_{n-1}^{(k)}c_{n+1}^{(k)}\\
&&\leq\bigg[\underbrace{\int_0^1\cdots \int_0^1}_k x_1\cdots x_k(1-x_1\cdots x_k)\cdots(n-1-x_1\cdots x_k)dx_1\cdots dx_k\bigg]^2\\
&&\quad-\underbrace{\int_0^1\cdots \int_0^1}_k x_1\cdots x_k(1-x_1\cdots x_k)\cdots(n-2-x_1\cdots x_k)dx_1d\cdots dx_k\\
&&\quad\times(n-1)\underbrace{\int_0^1\cdots \int_0^1}_k x_1\cdots x_k(1-x_1\cdots x_k)\cdots(n-1-x_1\cdots x_k)dx_1\cdots dx_k \\
&&=-\underbrace{\int_0^1\cdots \int_0^1}_k x_1\cdots x_k(1-x_1\cdots x_k)\cdots(n-1-x_1\cdots x_k)dx_1\cdots dx_k \\
&&\quad\times\underbrace{\int_0^1\cdots \int_0^1}_k (x_1\cdots x_k)^2(1-x_1\cdots x_k)\cdots(n-2-x_1\cdots x_k)dx_1\cdots dx_k\\
&&\leq0.
\end{eqnarray*}
Hence, sequence $\{c_n^{(k)}\}_{n\ge 2}$ is log-convex.

Recall the definition
\begin{eqnarray}
\omega_n^{(k)}=\sum_{m=0}^n\bigg[{n\atop m}\bigg]\frac{1}{(m+1)^k}. \label{zf-2.1}
\end{eqnarray}
It is easy to see that the sequence $\bigg\{\frac{1}{(m+1)^k}\bigg\}_{m\ge 0}$ is log-convex. By means of Lemma 2.1, we get that the sequence $\{\omega_n^{(k)}\}_{n\ge 0}$ is log-convex.
\end{proof}

\begin{thm}
For the sequence $\{\omega_n^{(k)}\}_{n\ge 3}$, we have
\begin{eqnarray}
\omega_n^{(k)}<\bigg[{n\atop K_n}\bigg]\sum_{m=1}^n\frac{1}{(m+1)^k}, \label{zf-2.2}
\end{eqnarray}
where $K_n$ is the index of the maximal unsigned Stirling numbers of the first kind $[{n\atop m}]$ for all fixed $n\ge 3$.
\end{thm}
\begin{proof}
For the Stirling numbers of the first kind, we know that
\begin{eqnarray*}
\bigg[{n\atop 1}\bigg]<\bigg[{n\atop 2}\bigg]<\cdots<\bigg[{n\atop K_n-1}\bigg]<\bigg[{n\atop K_n}\bigg]>\bigg[{n\atop K_n+1}\bigg]>\cdots>\bigg[{n\atop n}\bigg]\,.
\end{eqnarray*}
where $K_n\sim\frac{n}{\ln n} (n\to\infty)$ (see \cite{ref7}). We note that
\begin{eqnarray*}
\bigg[{n\atop 0}\bigg]&=&0, \quad n\ge 1, \\
\bigg[{n\atop m}\bigg]\frac{1}{(m+1)^k}&<&\bigg[{n\atop K_n}\bigg]\frac{1}{(m+1)^k}, \quad 1\le m\le n \quad (m\neq K_n),
\end{eqnarray*}
By applying (\ref{zf-2.1}), we obtain (\ref{zf-2.2}).
\end{proof}

We now consider the log-convexity and unimodality of multiparameter-poly-Cauchy numbers of two kinds under some conditions.

\begin{thm}
Assume that the sequence $A=(0, \alpha_1, \ldots, \alpha_n, \ldots)$ satisfies that $\alpha_j\ge 1$, and $\alpha_j-\alpha_{j-1}\ge 1$ for $j\ge 1$. The sequences $\{c_{n,A}^{(k)}\}_{n\ge 2}$ and $\{\widehat c_{n,A}^{(k)}\}_{n\ge 0}$ are log-convex.
\end{thm}
\begin{proof}
Since $\alpha_j\ge1$ and $\alpha_j-\alpha_{j-1}\ge1$ ($j\ge1$),
for $n\ge 3$ we have
\begin{eqnarray*}
&&\bigg(c_{n,A}^{(k)}\bigg)^2-c_{n-1,A}^{(k)}c_{n+1,A}^{(k)}\\
&&=\bigg[\underbrace{\int_0^1\cdots \int_0^1}_k x_1\cdots x_k(\alpha_1-x_1\cdots x_k)\cdots(\alpha_{n-1}-x_1\cdots x_k)dx_1\cdots dx_k\bigg]^2\\
&&\quad-\underbrace{\int_0^1\cdots \int_0^1}_k x_1\cdots x_k(\alpha_1-x_1\cdots x_k)\cdots(\alpha_{n-2}-x_1\cdots x_k)dx_1\cdots dx_k\\
&&\quad\times\underbrace{\int_0^1\cdots \int_0^1}_k x_1\cdots x_k(\alpha_1-x_1\cdots x_k)\cdots(\alpha_n-x_1\cdots x_k)dx_1\cdots dx_k,\\
&&=-\underbrace{\int_0^1\cdots \int_0^1}_k (x_1\cdots x_k)^2(\alpha_1-x_1\cdots x_k)\cdots(\alpha_{n-2}-x_1\cdots x_k)dx_1\cdots dx_k\\
&&\quad\times\underbrace{\int_0^1\cdots \int_0^1}_k x_1\cdots x_k(\alpha_1-x_1\cdots x_k)\cdots(\alpha_{n-1}-x_1\cdots x_k)dx_1\cdots dx_k\\
&&\quad-\underbrace{\int_0^1\cdots \int_0^1}_k x_1\cdots x_k(\alpha_1-x_1\cdots x_k)\cdots(\alpha_{n-2}-x_1\cdots x_k)dx_1\cdots dx_k\\
&&\quad\times\underbrace{\int_0^1\cdots \int_0^1}_k x_1\cdots x_k(\alpha_1-x_1\cdots x_k)\cdots(\alpha_{n-1}-x_1\cdots x_k)\\
&&\qquad\qquad\times(\alpha_n-\alpha_{n-1}-x_1\cdots x_k)dx_1\cdots dx_k\\
&&\le-\underbrace{\int_0^1\cdots \int_0^1}_k (x_1\cdots x_k)^2(\alpha_1-x_1\cdots x_k)\cdots(\alpha_{n-2}-x_1\cdots x_k)dx_1\cdots dx_k\\
&&\quad\times\underbrace{\int_0^1\cdots \int_0^1}_k x_1\cdots x_k(\alpha_1-x_1\cdots x_k)\cdots(\alpha_{n-1}-x_1\cdots x_k)dx_1\cdots dx_k\\
&&\le0,
\end{eqnarray*}

Similarly, for $n\ge 1$ we have
\begin{eqnarray*}
&&\bigg(\widehat c_{n,A}^{(k)}\bigg)^2-\widehat c_{n-1,A}^{(k)}\widehat c_{n+1,A}^{(k)}\\
&&=\bigg[\underbrace{\int_0^1\cdots \int_0^1}_k x_1\cdots x_k(x_1\cdots x_k+\alpha_1)\cdots(x_1\cdots x_k+\alpha_{n-1})dx_1\cdots dx_k\bigg]^2\\
&&\quad-\underbrace{\int_0^1\cdots \int_0^1}_k x_1\cdots x_k(x_1\cdots x_k+\alpha_1)\cdots(x_1\cdots x_k+\alpha_{n-2})dx_1\cdots dx_k \\
&&\quad\times\underbrace{\int_0^1\cdots \int_0^1}_k x_1\cdots x_k(x_1\cdots x_k+\alpha_1)\cdots(x_1\cdots x_k+\alpha_n)dx_1\cdots dx_k\\
&&\le-\underbrace{\int_0^1\cdots \int_0^1}_k (x_1\cdots x_k)^2(x_1\cdots x_k+\alpha_1)\cdots(x_1\cdots x_k+\alpha_{n-1})dx_1\cdots dx_k\\
&&\quad\times\underbrace{\int_0^1\cdots \int_0^1}_k x_1\cdots x_k(x_1\cdots x_k+\alpha_1)\cdots(x_1\cdots x_k+\alpha_{n-2})dx_1\cdots dx_k,\\
&&\le 0.
\end{eqnarray*}
Hence, the sequence $\{c_{n,A}^{(k)}\}_{n\ge 2}$ and $\{\widehat c_{n,A}^{(k)}\}_{n\ge 0}$ are log-convex.
\end{proof}

Note that Theorem 2.3 generalizes Theorem 2.1. Clearly, Theorem 2.3 becomes Theorem 2.1 when $A=(0, 1, 2, \ldots, n, \ldots)$.

\begin{thm}
Suppose that the sequence $A=(0, \alpha_1, \ldots, \alpha_n, \ldots)$ satisfies that $\alpha_j\ge0$ for $j\ge 1$. Then we have:
\\
\noindent
{\rm (i)} if there exists $l\ge 3$ such that $\alpha_j\ge 2$ for $1\le j\le l$ and $\alpha_j=1$ for $j\ge l+1$, then $\{\sigma_{n,A}^{(k)}\}_{n\ge 1}$ is unimodal, and its single peak is at $l+1$; \\
\noindent
{\rm (ii)} if there exists $l\ge 3$ such that $\alpha_j\ge 1$ for $1\le j\le l$ and $\alpha_j=0$ for $j\ge l+1$, then $\{\omega_{n,A}^{(k)}\}_{n\ge 1}$ is unimodal, and its single peak is at $l+1$.
\end{thm}
\begin{proof}
{\rm (i)} For $n\ge 1$,
\begin{eqnarray*}
&&\quad\sigma_{n+1,A}^{(k)}-\sigma_{n,A}^{(k)}\\
&&=\underbrace{\int_0^1\cdots \int_0^1}_k x_1\cdots x_k(\alpha_1-x_1\cdots x_k)\cdots(\alpha_{n-1}-x_1\cdots x_k)(\alpha_n-1-x_1\cdots x_k)dx_1\cdots dx_k.
\end{eqnarray*}
We can verify that $\sigma_{n+1,A}^{(k)}-\sigma_{n,A}^{(k)}\ge 0$ for $1\le n\le l$ and $\sigma_{n+1,A}^{(k)}-\sigma_{n,A}^{(k)}\le 0$ for $n\ge l+1$.
Hence, $\{\sigma_{n,A}^{(k)}\}_{n\ge 1}$ is unimodal, and its single peak is at $l+1$.

\noindent
{\rm (ii)} For $n\ge 1$,
\begin{eqnarray*}
&&\quad\omega_{n+1,A}^{(k)}-\omega_{n,A}^{(k)}\\
&&=\underbrace{\int_0^1\cdots \int_0^1}_k x_1\cdots x_k(x_1\cdots x_k+\alpha_1)\cdots(x_1\cdots x_k+\alpha_{n-1})(x_1\cdots x_k+\alpha_n-1)dx_1\cdots dx_k.
\end{eqnarray*}
We can verify that $\omega_{n+1,A}^{(k)}-\omega_{n,A}^{(k)}\ge 0$ for $1\le n\le l$ and $\omega_{n+1,A}^{(k)}-\omega_{n,A}^{(k)}\le 0$ for $n\ge l+1$. Therefore, $\{\omega_{n,A}^{(k)}\}_{n\ge 1}$ is unimodal, and its single peak is at $l+1$.
\end{proof}

\section{Acknowledgements}
The authors would like to thank an anonymous referee whose helpful suggestions
and comments have led to much improvement of the paper.

\end{document}